\documentclass[12pt,reqno]{amsart}

\usepackage{graphicx}
\usepackage{amsmath}
\usepackage{amssymb}
\usepackage{amscd}
\usepackage{hhline}
\usepackage{dcpic}
\usepackage{pictex}

\newtheorem{theorem}{Theorem}
\newtheorem{theorema}{Theorem}
\newtheorem{theoremb}{Theorem}

\newtheorem{theoremd}{Theorem}
\newtheorem{theoreme}{Theorem}
\newtheorem{dfn}[theoremb]{Definition}
\newtheorem{rk}[theoremd]{Remark}

\newtheorem{lem}[theoreme]{Lemma}

\newtheorem{prop}[theorem]{Proposition}
\newenvironment{Proof}[1]{\textbf{#1.} }

\newcommand\bib[1]{\bibitem[#1]{#1}}

\renewcommand\a{\alpha}
\renewcommand\b{\beta}

\newcommand\C{{\mathbb C}}
\newcommand\CC{{\mathcal C}}

\newcommand\D{{\mathcal D}}

\newcommand\E{\mathcal{E}}

\newcommand\g{\gamma}

\renewcommand\l{\lambda}

\newcommand\oo{\omega}
\newcommand\op[1]{\mathop{\rm #1}\nolimits}
\newcommand\ot{\otimes}
\newcommand\p{\partial}

\newcommand\R{{\mathbb R}}
\renewcommand\t{\tau}

\newcommand\ti{\tilde}

\newcommand\vp{\varphi}
\newcommand\we{\wedge}
\newcommand\z{\sigma}

\newcommand{\comm}[1]{}

\begin{document}

 \title[Lie theorem via rank 2 distributions]{Lie theorem via rank 2 distributions \\
 (integration of PDE of class $\oo=1$)}
 \author{Boris Kruglikov}
 %\date{}
 \address{Institute Mathematics and Statistics, NT-faculty, University of Troms\o, Troms\o\ 90-37, Norway}
 \email{boris.kruglikov@uit.no}

 \vspace{-14.5pt}
 \begin{abstract}
In this paper we investigate compatible overdetermined systems of
PDEs on the plane with one common characteristic. Lie's theorem
states that its integration is equivalent to a system of ODEs, and
we relate this to the geometry of rank 2 distributions. We find a
criterion for integration in quadratures and in closed form, as in
the method of Darboux, and discuss nonlinear Laplace transformations
and symmetric PDE models.
 \end{abstract}

 \maketitle

%%%%%%%%%%%%%%%%%%%%%%%%%%%%%%%%%%%%%%%%%%%%%%%%%%%%%%%%%%%%%%%%%%%%%%%%%%%%
%0%
\section*{Introduction}

Consider a system $\E$ of partial differential equations of (maximal) order $k$
(geometrically a submanifold in $k$-jets).
We will study its integration micro-locally near a regular point
(in a neighborhood $U\subset\E$).

We will mainly restrict to the case of two independent and one dependent variables
(non-scalar systems can be treated similarly or re-written via Drah's trick \cite{Sto};
the case of more independent variables makes no fundamental difference either).

%%%%%%%%%%%%
\subsection{Formulation of the problem}\label{S01}

We assume the system $\E$ is overdetermined and is compatible (formally integrable).

 % because otherwise one should add compatibility conditions before any
 % further investigations.
 % The closure process for an arbitrary regular system consists of a
 % finite number of steps. This is the context of the celebrated
 % Cartan-Kuranishi prolongation-projection method.

Whenever the above assumptions on $\E$ are fulfilled we can evaluate
formal dimension and rank of the system (like in Cartan test \cite{C$_3$}),
namely determine on how many functions of how many variables a
general solution formally depends.

For a (determined or overdetermined) system on the plane two situations are possible. Either the
system has empty characteristic variety or the latter is an
effective divisor on $\C P^1$ (Weil divisor in algebraic geometry;
here a finite set of points with multiplicities).

In the first case the system $\E$ is of finite type, i.e. its local
solution space is finite dimensional. Then integration of $\E$ can
be reduced to a system of ODEs; we call such systems Frobenius.

In the second case additional assumptions must be imposed to
guarantee existence of solutions. Counting of solutions can still be
carried on the formal level and the general stratum of the solution
space of $\E$ is parametrized by $\oo$ functions of 1
variable\footnote{Sophus Lie called the number $\oo$ class of the
system; for Ellie Cartan this is the character $s_1$ (provided the
Cartan number is 1: $s_i=0$ for $i>1$).}, where $\oo$ is the degree of
the divisor, i.e. the number of points counted with multiplicities.

Note that $\oo$ can be described in a different way: Since the
characteristic variety is discrete, the symbol $g_k$ of the system
$\E$ stabilizes and its dimension is $\dim g_k=\oo$ for big $k$.

In this paper we restrict to the case $\oo=1$, which is the next
simplest after the finite type case $\oo=0$.

%%%%%%%%%%%%
\subsection{Main results}\label{S02}

For $\oo=1$ Sophus Lie obtained in 1893 a theorem, which states that this
case can be reduced to ODEs as well:

 \begin{theorema}\label{thA}
A compatible regular overdetermined system $\E$ of class $\oo=1$ can
be locally integrated via ODEs.
 \end{theorema}

The proof in \cite{L} is rather sketchy. The result was later obtained
in \cite{De} without reference to Lie. We will demonstrate the
claim via geometry of rank 2 distributions and relate it to other
important results\footnote{Let us mention that paper \cite{Y} discusses another
reduction to ODEs for the involutive PDE systems of the 2nd order. This family meets
ours by class $\oo=1$ systems of type $2E_2$ in terminology of \S\ref{S21},
which were studied by Cartan in \cite{C$_1$}.}.

Three remarks are of order. First is that we can change here formal
integrability to local integrability. Generally it is wrong due to
Levi and similar examples. The essential feature is regularity of
the characteristic variety and peculiarity of $\oo=1$ (i.e. for $\oo>1$ the passage
from formal integrability to local integrability is not generally correct).

Second, and more important, is that the reduction procedure can be
made explicit and this allows not only to claim the reduction
theoretically, but also to develop a practical algorithm for
integrability.

And third is that facing only the problem of reduction, we can allow
arbitrary number of independent and dependent variables, leaving
only the requirements of regularity and $\oo=1$, i.e. precisely one
common characteristic (counted with multiplicity) for the equations
of $\E$.

In this paper we mostly concentrate on the problem of effective
integration of PDE systems. So we are especially interested in quadrature
and integration in closed form (an ideal case of Darboux integrability --
the definitions are below in the text, see also \cite{Da,F,AFV}).

 \begin{theorema}\label{thB}
A system $\E$ of class $\oo=1$ is integrable in closed form iff it is
linearizable by an internal transformation.

A system $\E$ of class $\oo=1$ is integrable in quadratures iff it has
a transitive solvable Lie algebra of internal symmetries.\footnote{Internal symmetries 
are transformations of the equation $\E$ considered as a manifold preserving the induced 
Cartan distribution. They are more general than the classical Lie symmetries, but can 
differ from the higher (Lie-B\"acklund) symmetries.}
 \end{theorema}

Of course, one is interested in algorithmic integration, so that an effective linearization is
important. Then a sequence of generalized Laplace transformations (these are the external 
transformations introduced for class $\oo=1$ in \cite{K}) finishes the job.

 %  \begin{rk}
 % In this theorem we assume that the only characteristic $X$
 % for $\oo=1$ system is straightened. This amounts to certain ODE integration,
 % but most problems in integrating PDE systems of class $\oo=1$ is related to other issues,
 % and we specify when they are gone.
 %  \end{rk}

As we shall explain, determining both linearization and quadrature is related to investigation
of the rank 2 distributions internally related to the system $\E$.
Linearizable systems correspond to Goursat distributions, i.e.
canonical Cartan distributions on the jet spaces for ODEs (in
general non-linear situation the growth vector is un-restricted).
Rank 2 distributions for the systems integrable in quadratures have the structure of
integrable extensions, which can be decoded starting from its Tanaka algebra.

Thus we can model types of reduction, based on the normal forms of
rank 2 distributions. In particular, the simplest among exactly solvable
class $\oo=1$ compatible non-linearizable PDE systems will be those that can be
reduced to Hilbert-Cartan equation (it's symmetry algebra will contain the exceptional
Lie group $G_2$). More complicated examples will be presented at the end of the paper.

%%%%%%%%%%%%
\subsection{Structure of the paper}\label{S03}

We will exploit the geometric theory of PDE, jet-geometry and the basics of
Spencer formal theory. We will also use the geometry of vector distributions.
The reader is invited to consult \cite{S,T,KLV} for details.

Notations are different from source to source, and we adapt those of \cite{KL}.
Since this paper is a continuation of \cite{K}, an acquaintance with the latter
will be useful (but not mandatory).

The paper is organized as follows.

In Section \ref{S1} we recast the class $\oo=1$ systems into the language of the geometry
of differential equations and provide a new modern proof of Theorem \ref{thA}.
Reduction to rank 2 distributions is the crucial ingredient\footnote{There is no
difference in arguments and in Section \ref{S1} we take dimension of the base $\R^n(x)$ arbitrary.
In further sections we restrict to $n=2$ independent variables for simplicity of exposition.}.
We then discuss an algorithmic method to integrate such systems and prove the
1st half of Theorem \ref{thB}.

In Section \ref{S2} we discuss another more general method of integration of PDEs via
integrable extensions (coverings), and relate this to the generalized symmetries.
Notice that integrable extensions for rank 2 distributions were
classified in \cite{AK}, so their description in the symmetric cases reduce to
purely algebraic questions.

In Section \ref{S3} we formulate the main invariants of compatible systems $\E$ of
class $\oo=1$, and we discuss transformations of such systems in linear and non-linear
cases. We investigate linear system from the viewpoint of internal geometry
(complimentary to the external point of view in \cite{K}), obtain the linearization criterion
and finish the proof of Theorem \ref{thB}.
Depending on the type of the system and its reduced rank 2 distribution we can
describe the structure of the general solution and a method of its integration.

Section \ref{S4} is devoted to various examples of compatible PDE systems
of class $\oo=1$. We will perform integration via the method of integrable 
extensions, generalized nonlinear Laplace transformations and discuss 
their relation to Darboux integrability. Some of the most symmetric examples
are coverings of the overdetermined involutive system of order 2 on the plane 
investigated by E.Cartan.

\medskip

\textsc{Acknowledgment.} It is a pleasure to thank Nail Ibragimov for his translation of
Sophus Lie paper \cite{L} in \cite{LG}, which was a starting point for this paper.
I am grateful to Valentin Lychagin for many discussions on the initial stage of the project.
The paper is strongly influenced by a collaboration with Ian Anderson, to whom I am thankful.

Hospitality of MLI (Stockholm) in 2007, MFO (Oberwolfach) in 2008,
Banach center (Warsaw) and Utah State University in 2009, and IHES (Paris) in 2010
is highly acknowledged.

%%%%%%%%%%%%%%%%%%%%%%%%%%%%%%%%%%%%%%%%%%%%%%%%%%%%%%%%%%%%%%%%%%%%%%%%%%
\section{Around Sophus Lie theorems}\label{S1}

In this section we give a modern proof of Theorem \ref{thA}.
Sophus Lie's original approach is indirect and hard to implement.
 % Indeed consider the $2E_2$ $\{r=s=0\}$. By Lie's method we seek i.i. $V(x,y,u,p,q)=c(y)$,
 % i.e. after D_x-prolongation we get $V_x+pV_u=0$. This has solution $V=v(u-xp,y,p,q)$, which
 % does not lead to the solution of Hamilton-Jacobi equation, contrary to the direct method $u=cx+\psi(y)$
We present a geometric method, which is the base of our approach to integration of class $\oo=1$ systems.
Furthermore we will elaborate this theorem to get the constructive Theorem \ref{thB}.

 % In this section only we allow the general situation of arbitrary number of
 % independent and dependent variables.

%%%%%%%%%%%%
\subsection{The geometric setup}\label{S11}

Consider a compatible overdetermined system of PDEs as a submanifold in the space of jets
$\E\subset J^k(W,N)$, where $W=\R^n$ is the space of independent variables $x=(x^i)$,
$N=\R^m$ is the space of dependent variables $u=(u^j)$
(assuming $\E$ to be of pure order $k$ is not crucial).
Let $\pi_k:J^k\to W$, $\pi_{k,k-1}:J^k\to J^{k-1}$ denote the natural projections.

We prolong the system to the level it becomes involutive (see the discussion about
relation of this with compatibility in \cite{KL} and \cite{K}). The assumption on the class $\oo=1$ yields $\dim g_k=1$ for
the symbol space starting from this level (still denoted by $k$).

This jet-space is equipped with the canonical Cartan distribution $\mathcal{C}=\op{Ann}(\theta_\sigma^j:|\sigma|<k)\subset TJ^k$
(where $\theta_\sigma^j=du_\sigma^j-\sum u_{\sigma+1_i}^jdx^i$ in canonical coordinates).
The induced Cartan distribution on the equation $\mathcal{C}_\E=\mathcal{C}\cap T\E$ has rank $n+1$.
Indeed, it is generated by one vertical vector -- a generator of the symbol $g_k=\op{Ker}(d\pi_{k,k-1}:T\E\to TJ^{k-1})$,
which is defined up to scaling, and $n$ total derivatives
$\D_{x^i}=\p_{x^i}+\sum u_{\sigma+1_i}^j\p_{u_\sigma^j}$ restricted to the equation, which are defined mod $g_k$.
Denoting the horizontal space by $H$ (not canonical) we get
 $$
\mathcal{C}_\E=H\oplus g_k.
 $$

 \begin{lem}\label{LemChar}
Let $p\in PT^*W$ be the (unique) characteristic covector. There is a unique $(n-1)$-dimensional subdistribution
$\Pi\subset\mathcal{C}_\E$ such that $d\pi_k(\Pi)=p^\perp\subset TW$ and $\Pi$ consists of Cauchy characteristics
of $\mathcal{C}_\E$.
 \end{lem}

 \begin{proof}
Let $H$ be some choice of horizontal space, $\Pi\subset H$ the lift of $p^\perp$
and $\eta$ a vertical vector field (section of $g_k$).

Let us use the standard identification $\pi_{k,k-1}^{-1}(*)\simeq S^k T^*W\ot TN$ \cite{KLV}.
Then the condition $\oo=1$ translates to $g_k=\langle p^k\ot v\rangle$ for some $v\in TN$.
The Lie bracket induces a pointwise bracket
$H\otimes g_k\to\nu_{k-1}=\langle p^{k-1}\ot v\rangle\subset S^{k-1} T^*W\ot TN$.

With identification $H\simeq TW$ this latter is the restriction of the natural pairing
$TW\otimes S^k T^*W\ot TN\to S^{k-1} T^*W\ot TN$. It follows that $[\xi,\eta]=0\,\op{mod}\mathcal{C}_\E$
and $[\xi,\xi']=0\,\op{mod}\mathcal{C}_\E$ $\forall$ $\xi,\xi'\in\Pi$.

 % Since the total derivative operators commute, the two vector fields from $\Gamma(\tilde\Pi)$
 % commute mod $\mathcal{C}_\E$. Indeed, such fields can be taken restrictions of the total derivatives
 % to the equations, and they are defined mod $g_k$.

It remains to choose an additional vector field $\zeta$ in $H\setminus\Pi$ and consider the induced bracket
$\tau:\langle\zeta\rangle\otimes\Pi\to\nu_{k-1}$. It can be non-zero since the pairing
$(\zeta,p^k\otimes\theta)\mapsto k\,p^{k-1}p(\zeta)\otimes\theta\ne0$.

Let us change $H=\langle\zeta\rangle\oplus\Pi$ by modifying $\Pi$ as the graph of the map
$-\tau(\zeta,\cdot)\in\Pi^*\ot\nu_{k-1}\simeq\Pi^*\ot g_k$ (identification
via $L_\zeta|_{g_k}$). Then the new space $\Pi$ is still involutive
with respect to the induced bracket and it commutes with both $\zeta$ and $\eta$ mod $\mathcal{C}_\E$.
This means that the sections of $\Pi$ are Cauchy characteristics.

Uniqueness of $\Pi$ follows from the fact that the above (bracket)
pairing $\langle\zeta\rangle\otimes g_k\to\nu_{k-1}$ is non-zero.
 \end{proof}

 \begin{rk}
For $n=2$ characteristic vectors are dual to characteristic covectors. It is not however true that
the former can be lifted to Cauchy characteristics of $\mathcal{C}_\E$. This is peculiarity of
the case $\oo=1$.
 \end{rk}

 % Let us notice that we proved integrability of characteristics in an elementary fashion in our case
 % (compare \cite{GQS,Ga}).

%%%%%%%%%%%%
\subsection{Reduction to rank 2 distributions}\label{S12}

Due to Lemma \ref{LemChar} internal geometry of the distribution $\mathcal{C}_\E$
is equivalent to that of the rank 2 distribution $\mathcal{C}_\E/\Pi$.
This implies Sophus Lie theorem:

\smallskip
 \begin{Proof}{Proof of Theorem \ref{thA}}
Consider the pair $(\E,\mathcal{C}_\E)$. Solutions of the system are
$n$-dimensional integral submanifolds of the distribution, whose
projection to the base are submersive.

It follows from the proof of Lemma \ref{LemChar} that an $n$-dimensional subspace of $\mathcal{C}_\E$
is involutive with respect to the (bracket) pairing $\Lambda^2\mathcal{C}_\E\to\nu_{k-1}$ iff
it contains $\Pi$. In other words, a solution must be tangent to $\Pi$.
 % and be transversal to $g_k$ (this latter requirement can be omitted for generalized solutions).

It is the standard fact, that the sub-distribution $\Pi$ generated by Cauchy characteristics
is integrable and shifts along it are symmetries for $\mathcal{C}_\E$.
Taking the (local) quotient we arrive to the manifold $M=\E/\Pi$ (quotient by the leaves)
equipped with a rank 2 distribution $\Delta=\mathcal{C}_\E/\Pi$ without characteristics.

 % General theory of rank 2 distributions is known since Goursat, Vessiot and Cartan, but
 % we need only the fact that such a distribution has integral curves, which is an elementary matter.

Such a distribution has integral curves, which are found by solving underdetermined ODEs.
The space of integral curves is locally parametrized by 1 arbitrary function of 1 variable
(determinization of the ODE, given e.g. by the choice of a curve in the image of any submersion
$M\to\R^2$ with fibers transversal to $\Delta$). The inverse of the quotient map $\E\to M$ sends
any of them to an $n$--dimensional integral surface, i.e. the solutions of $\E$. \qed
 \end{Proof}

\smallskip

Note that this proof, as well as the arguments from the previous
subsection, uses integration of ODE systems twice: first to solve
the Frobenius system, corresponding to Cauchy characteristics $\Pi$,
and then to find the integral curves of $\Delta=\mathcal{C}_\E/\Pi$.

The latter integration can be split in turn into integration of the
bracket-closure of the distribution $\Delta_\infty=\Delta+[\Delta,\Delta]+\dots$,
which is Frobenius and then integrating $\Delta$ in the leaves.

In the first case the order of the system is $\dim\E-(n-1)=\op{codim}\Pi$.
In the second it is split into an ODE of order equal to the number of first integrals for $\Delta$
in $M$ ($=\op{codim}\Delta_\infty$) and an ODE of order $\op{dim}\Delta_\infty-2$.

%%%%%%%%%%%%
\subsection{Constructive integration methods}\label{S13}

A theorem of Sophus Lie states that ODEs with a transitive solvable Lie algebra of symmetries
are integrable in quadratures. This is equivalent to the claim that if a holonomic
distribution\footnote{This means it satisfies the Frobenius condition $[\Gamma(\Delta),\Gamma(\Delta)]\subset\Gamma(\Delta)$.} $\Delta$ on a manifold $M$
admits a solvable symmetry Lie group of complimentary dimension
with orbits transversal to it, then the integral leaves of $\Delta$ can be expressed in
quadratures \cite{KLR}.

We extend this theorem to non-holonomic distributions $\Delta$.
We assume at first the distributions are completely non-holonomic, i.e. the bracket closure
$\Delta_\infty$ equals $TM$ and so $\Delta$ has no first integrals.
Generic such distributions have no integral surfaces, and integral curves
(which always exist) are the maximal integral manifolds.

 \begin{theorem}\label{Thm1}
Let $\Delta$ be a completely non-holonomic distribution of rank $r$ on a manifold $M^n$.
Suppose a solvable Lie group $G$ of dimension $n-r$ acts by symmetries with orbits everywhere transversal
to $\Delta$. Then local integral curves of $\Delta$ can be found by quadratures.
 \end{theorem}

 \begin{proof}
Denote $\pi:M\to L^r=M/G$ the local quotient by the orbits.
(the space of $G$-invariants). Notice that $\pi_*$ maps $\Delta$ to $TL$.

Choose a curve $\g\subset L$ and restrict the distribution $\Delta$ to $\pi^{-1}(\g)$.
This is a line field and $G$ acts transitively by symmetries on $\pi^{-1}(\g)$.
By the classical Sophus Lie theorem the integral curves of this line field can be found by
quadratures. Thus these restricted integral curves are parametrized by $n-r$ integration
constants in $\pi^{-1}(\g)$, while the curves $\g\subset L$ are parametrized
by $r-1$ function of 1 variable. The integral curves of $\Delta$ in $M$ are given
through these by quadratures.
 \end{proof}

In particular, for our case $r=2$ we get dependence on 1 function of 1 variable.
Thus for general class $\oo=1$ compatible PDE system we need three solvable Lie group
to integrate it in quadratures: one group $G_1$ of dimension equal to corang of the
characteristic space $\Pi$ to perform the reduction $(\E,\mathcal{C}_\E)\to(M,\Delta)$,
the second group $G_2$ of dimension equal to corang of $\Delta_\infty$ in $M$,
and finally the third group $G_3$ of dimension $\op{rank}(\Delta_\infty)-\op{rank}(\Delta)$
(all actions should be transversal).

 \begin{rk}\label{afterThm1}
A more general result is this: Consider a solvable Lie algebra $G$ acting as transversal
symmetries of $\Delta$ in $M$. Denote by $\bar\Delta$ in $\bar M$ the quotient distribution.
Then integral curves of $\Delta$ can be found from integral curves of $\bar\Delta$ by quadratures.
The number of involved integrals in the formula for the general integral curve
is equal to the length of the derived series of $G$.
 \end{rk}

Let us consider an example from \cite{Str} of a Monge equation $\E$
on $y=y(x)$, $z=z(x)$ with 3-dimensional solvable symmetry group:
 \begin{equation}\label{FSz}
z'=z^2+\psi(z)+(y''+y)^2.
 \end{equation}
The group $G=\op{Sym}(\E)$ is generated by the (prolongations of) vector fields
$\p_x,\cos x\,\p_y,\sin x\,\p_y$ on $J^0(\R,\R^2)=\R^3(x,y,z)$.

If one (naively) substitutes $y=h(x)$, then $z(x)$ satisfies a Riccati equation,
and so its solution cannot be found by quadratures (a similar problem occurs for
general $\oo=1$ class PDEs, so general reduction to ODEs from Theorem \ref{thA}
does not necessarily yields a solution).

The correct approach of Theorem \ref{Thm1} is to consider the quotient, i.e.
to pass to the space of $G$-invariants $\R^2(z,y''+y)$.
A curve in this space is given by an equation $y''+y=f(z)$.
Substituting this back into (\ref{FSz}) we find the autonomous first order equation
 $$
z'=z^2+\psi(z)+f(z)^2,
 $$
which is easily integrable in quadratures.

 \begin{rk}
The previous naive argument uncover as follows. The curve in the plane $L^2=\R^2(z,y''+y)$
is specified via a parameter $x$: $y''+y=\tilde h(x)$ and $z(x)$ is given by
$z'=z^2+\psi(z)+\tilde h(x)^2$. Since the last equation cannot be integrated in quadratures,
the initial data (curve in $L^2$, a point over it determined by 3 constants) is not given explicitly,
and so the result ceases to be given via an explicit formula.
 \end{rk}

%%%%%%%%%%%%%%%%%%%%%%%%%%%%%%%%%%%%%%%%%%%%%%%%%%%%%%%%%%%%%%%%%%%%%%%%%%
\section{Integrable extensions and generalized symmetries}\label{S2}

Integrable extensions or coverings \cite{KV} are mappings of PDEs $\E\to\bar\E$
such that solutions of $\E$ are obtained from those of $\bar\E$ by solving ODEs.
For (underdetermined) ODEs the covering is just a submersion $\pi:(\tilde M,\tilde\Delta)\to(M,\Delta)$,
i.e. $d_x\pi:\tilde\Delta_x\to\Delta_{\pi(x)}$ is an isomorphism for any $x\in\tilde M$.

These coverings of systems of ODEs (or distributions) were studied in \cite{AK} as they are useful
in solving the system. Indeed a sequence of integrable extensions can decompose a given system
into a sequence of 1st order scalar ODEs.

It is easy to see that quotient by the Cauchy characteristic of class $\oo=1$ systems,
which is basic for Theorem \ref{thA}, commutes with integrable extension. Thus it is enough to
study integrable extensions of rank 2 distributions.
We will relate them to the symmetry approach of the previous section.

For instance, we can write the symmetry reduction of
Theorem \ref{Thm1} and Remark \ref{afterThm1} via integrable extensions. Let $\rho:\bar M\to\R(x)$
be function (projection) giving the independence condition. Write the equation for integral curves
of $\bar\Delta$ as $F[x,u]=0$, where the latter is an ordinary (nonlinear) differential operator and
both $F$ and $u$ are multi-dimensional.

Then provided $G$ has derived series $G=G_l\supset G_{l-1}\supset\dots \supset G_0=0$
with Abelian quotients $G_i/G_{i+1}=V_i$ we can choose a coordinate $v_i$ on $V_i$ and
have the equation for integral curves of $\Delta$ in this form:
 \begin{multline*}
F[x,u]=0,\ v_1'=H_1[x,u],\ v_2'=H_2[x,u,v_1],\ \dots,\\
v_{l-1}'=H_{l-1}[x,u,v_1,\dots,v_{l-2}],\ v_l'=H_l[x,u,v_1,\dots,v_{l-1}].
 \end{multline*}

%%%%%%%%%%%%
\subsection{Existence of integral de-prolongations for $(2,5)$ distributions}\label{S21}

Due to existence of normal forms $(2,n)$ distributions have the structure of integrable extensions
for $n<5$. This holds true also in the first non-trivial case $n=5$, where such distributions
have moduli.

 \begin{theorem}
A regular 2-distribution $\Delta$ on a manifold $M^5$ admits local submersion onto
a 2-distribution in a 4-dimensional manifold $(\bar M^4,\bar\Delta)$.
 \end{theorem}

The claim follows from (is equivalent to) a result due to Goursat:

 \begin{theorem}[\cite{G$_3$}, \S 76]
A regular rank 2 distribution in a 5-dimensional manifold can be locally represented
as the canonical distribution of the Monge equation $\E:v'=f(x,u,u',u'',v)$.
 \end{theorem}

Indeed, the distribution of this equation $\E\subset J^{1,2}(\R,\R^2)$ is
 \begin{equation}\label{DeltaforMonge}
\Delta=\langle\p_x+u_1\,\p_u+u_2\,\p_{u_1}+f\,\p_v,\p_{u_2}\rangle,
 \end{equation}
which has the structure of integrable extension over $J^2(\R,\R)$, equipped with the canonical
Cartan distribution $\langle\p_x+u_1\,\p_u+u_2\,\p_{u_1},\p_{u_2}\rangle$; the projection $\pi:\E\to J^2$
is $(x,u,u_1,u_2,v)\mapsto(x,u,u_1,u_2)$.

For completeness we give an alternative proof of Goursat's theorem
(using vector fields approach instead of EDS methods).

 \begin{proof}
Let the flag of the distribution\footnote{By the commutator of two distributions we mean the
distribution generated by the commutator of sections, e.g. $\Gamma(\Delta_2)=[\Gamma(\Delta),
\Gamma(\Delta_1)]$ etc.} be $\Delta_1=\Delta$, $\Delta_2=[\Delta,\Delta_1]$,
$\Delta_3=[\Delta,\Delta_2]=TM$ (we consider the general situation; in the other cases the
distributions have normal forms and the statement follows).

Consider the maps $\Upsilon:\Gamma(\Delta)\times\Gamma(\Delta)\to\Gamma(\Lambda^4TM)$
and $\Theta_i:\Gamma(\Delta)\times\Gamma(\Delta)\to\Gamma(\Lambda^5TM)$ given by
 \begin{gather*}
\Upsilon(\zeta,\eta)=\zeta\we\eta\we[\zeta,\eta]\we[\zeta,[\zeta,\eta]],\\
\Theta_0(\zeta,\eta)=\Upsilon\we[\eta,[\zeta,\eta]],\quad
\Theta_1(\zeta,\eta)=\Upsilon\we[\zeta,[\zeta,[\zeta,\eta]]],\\
\Theta_2(\zeta,\eta)=\Upsilon\we[\zeta,[\eta,[\zeta,\eta]]],\quad
\Theta_3(\zeta,\eta)=\Upsilon\we[\eta,[\eta,[\zeta,\eta]]].
 \end{gather*}

A change of frame $\tilde\zeta=a\zeta+b\eta$, $\tilde\eta=c\zeta+d\eta$ induces the changes:
 $$
\Theta_0(\tilde\zeta,\tilde\eta)=\delta^5 \Theta_0(\zeta,\eta),\quad
\delta=\begin{vmatrix}
a & b \\ c & d\end{vmatrix}
 $$
 \begin{multline*}
\delta^{-4}\Theta_1(\tilde\zeta,\tilde\eta)=
a^3\,\Theta_1(\zeta,\eta)+a^2b\,(2\Theta_2(\zeta,\eta)+\Theta_3(\eta,\zeta))\\
+ab^2\,(2\Theta_2(\eta,\zeta)+\Theta_3(\zeta,\eta))+b^3\,\Theta_1(\eta,\zeta)
+\sigma\Theta_0(\eta,\zeta).
 \end{multline*}
where $\sigma=a\cdot(a\eta+b\zeta)(b)-b\cdot(a\eta+b\zeta)(a)$.
This implies existence of a solution $\frac ab\in C^\infty(M,\R P^1)$ to $\Theta_1(\zeta,\eta)=0$.

Let us straighten $\zeta=\p_{u_2}$ in a local chart $\R^5\hookrightarrow M$,
and denote the quotient by $\R^4=\R^5/\zeta$ (i.e. $u_2=\op{const}$).
Then the distribution becomes a $u_2$-dependent vector field $\eta=\Delta/\zeta$ in $\R^4$.
The Lie derivative $L_\zeta$ corresponds to the derivative by $u_2$, which we denote by the prime.

Condition $\Theta_1(\zeta,\eta)=0$ reads $\eta\we\eta'\we\eta''\we\eta'''=0$, and we can assume
the highest derivative can be resolved:
 $$
\eta'''=a_2\eta''+a_1\eta'+a_0\eta
 $$
By reparametrization of time $u_2$ and scaling of $\eta$ we can achieve $a_0=a_1=0$ (in contrast
the Laguerre-Forsyth canonical form). Then the equation is $\eta'''=a_2\eta''$,
and the solution is $\eta=\xi_0+u_2\xi_1+f\p_v$, where $f_{u_2u_2}\ne0$, $\eta''\|\p_v$ and
$\xi_0,\xi_1$ are $u_2$-independent fields on $\R^3=\R^4/\p_v$.

Now in our general case the distribution $\langle\xi_0,\xi_1\rangle$ in $\R^3$ is contact, so in
the Darboux coordinates $\xi_0=\p_x+u_1\p_u$, $\xi_1=\p_{u_1}$. Thus we obtain local coordinates
on $M$ such that $\Delta$ has form (\ref{DeltaforMonge}).
 \end{proof}

%%%%%%%%%%%%
\subsection{Non-existence of integral de-prolongations for $(2,n)$ distributions with $n>5$}\label{S22}

Dimensional count: generic rank 2 distribution in $M^n$ depends on $2(n-2)-n=n-4$ functions of $n$ variables
(quotient of sections of $(2,n)$-Grassmanian by the pseudogroup of local diffeomorphisms), while
integrable extension depends on 1 function of $n$ variables (in both cases: and some number of functions
of fewer variables). Thus for $n>5$ there are obstructions to existence of the structure of integrable
extension over a lower-dimensional manifold.

These obstructions are important relative differential invariants of the distribution.
For example, in dimension 6 there are 2 relative invariants vanishing of which characterizes
possibility to represent the distribution as the Monge equation $v'=f(x,u,u',u'',u''',v)$).

It is interesting to notice that for $n=5$ we get seemingly determined system (the same functional dimension).
And indeed, the one-dimensional distribution $V$ has the property of de-prolongation
(the vertical distribution of the projection $\pi$) iff $[V,\Delta]\subset V+\Delta$.
This writes as 4 equations on 4 functions specifying $V$ (these latter can be taken as the 1st integrals, but then
the system has order 2; it is better to write $V$ via a generating vector field $\p_v+F_1\p_x+F_2\p_u+F_3\p_p+F_4\p_q$
and take the components $F_i$ as the unknowns). This $4\times4$ system is not however determined: it has every covector
characteristic (direct calculation or the following observation -- in the normal form of the previous section
the totality of integrable de-prolongations has functional moduli - the general solution depends
on a function of $5$ variables).

 \begin{rk}
It is also interesting to try de-prolongation by rank 2 foliation, which indeed exists as a generic
(2,4) distribution has Engel normal form, and so integrally de-prolongs to the contact (2,3) distribution.
The conditions for existence of such rank 2 distribution $V$ are:
 \begin{equation}\label{VgsD}
[V,V]\subset V,\  [V,\Delta]\subset V+\Delta.
 \end{equation}
This is a system of 4 equations on 3 unknowns (1st integrals of $V$), but it is not overdetermined: again
all covectors are characteristic!
 \end{rk}

%%%%%%%%%%%%
\subsection{Generalized symmetries}\label{S23}

A symmetry of the distribution $\Delta$ is a vector field $\xi$ such that $\mathcal{L}_\xi(\Delta)=\Delta$.
More generally a space $\mathcal{G}$ of vector fields is a Lie symmetry algebra iff
 $$
[\mathcal{G},\mathcal{G}]\subset\mathcal{G},\ [\mathcal{G},\Delta]\subset\Delta.
 $$

 \begin{rk}
Let us recall that a collection of differential operators $\langle F_i\rangle$ form a symmetry algebra for
the PDE system $\E=\langle H_j=0\rangle$ iff
 $$
\{F_\a,F_\b\}=0\,\op{mod}\langle F_i\rangle,\ \ \{F_\a,H_\b\}=0\,\op{mod}\langle H_j\rangle,
 $$
where $\{,\}$ is the Jacobi bracket (we write the condition for simplicity in the case of scalar or
square matrix equations), see \cite{KLV}.

If we are interested in compatibility of the systems $\langle F_i\rangle$ and
$\langle H_j\rangle$, then the last condition changes to more general
 $$
\{F_\a,H_\b\}=0\,\op{mod}\langle F_i,H_j\rangle,
 $$
see \cite{KL}. Such $F$ are called generalized symmetries, conditional symmetries or auxiliary integrals.
 \end{rk}

Basing on this remark we can treat distributions $V$ satisfying condition (\ref{VgsD})
as generalized symmetries. It allows the following symmetry reduction: if $L\subset M$
is an integral curve of $\Delta$, then in the union of $V$-leaves meeting $L$
the integral curves of $\Delta$ can be found via lower-dimensional determined ODE
(for integral manifolds the corresponding criterion is a bit more complicated). If,
in addition, $V$ is obtained as the tangent distribution of solvable Lie algebra sheaf,
the solutions can be found in quadratures.
The generalized symmetries are more common than the classical ones.

 \smallskip

{\bf Example.} Consider the symmetries of the Engel distribution, which is the Cartan (higher contact)
distribution on $J^2(\R,\R)=\R^4(x,y,y_1,y_2)$. In canonical coordinates it is
$\Delta^2=\langle\xi_1=\p_x+y_1\p_y+y_2\p_{y_1},\xi_2=\p_{y_2}\rangle$.

By Lie-B\"acklund theorem the symmetries are lifts of contact fields on $J^1(\R,\R)$,
and so are defined by 1 function of 3 arguments.

The generalized symmetries $\eta=\p_y+\l_1\p_x+\l_2\p_{y_1}+\l_3\p_{y_2}$
(here unlike for symmetries we can normalize one of the coefficients by scaling)
are defined by $[\xi_i\,\eta]=0\,\op{mod}\Delta+\langle\eta\rangle$ which is equivalent to
 \begin{multline*}
\l_{2y_2}=\frac{y_1\l_2-y_2}{y_1\l_1-1}\l_{1y_2},\
\l_3=\frac{y_1\l_2-y_2}{1-y_1\l_1}\xi_1(\l_1)+\xi_1(\l_2)+\l_2\frac{\l_2-y_2\l_1}{1-y_1\l_1}.
 \end{multline*}
So the generalized symmetries depend on 1 function of 4 arguments $\l_1$.

 \smallskip

It is often the case that a system (distribution) has no symmetries, but it admits
generalized symmetries that can (partially) integrate $\Delta$.

%%%%%%%%%%%%%%%%%%%%%%%%%%%%%%%%%%%%%%%%%%%%%%%%%%%%%%%%%%%%%%%%%%%%%%%%%%
\section{Integration of class $\oo=1$ systems}\label{S3}

In this section we split the totality of $\oo=1$ systems into classes, and
discuss transformations between them as a method of integration.
$r=\dim H^{*,1}(\E)$ will be the total amount of PDEs in the system.\footnote{Starting
from this Section we restrict to base dimension $n=2$. Familiarity with
the Spencer cohomology $H^{i,j}(\E)$ \cite{S} is not crucial.}

%%%%%%%%%%%%
\subsection{Type and complexity}\label{S31}

We introduce the following rule for a choice of generators of the system $\E$ of class $\oo=1$.
Consider the orders of the PDEs in the system: $k_\text{min}=k_1\le\dots\le k_r=k_\text{max}$,
which are taken with multiplicities $m_i=\{\# j: k_j=i\}=\dim H^{i-1,1}(\E)$.

So the system $\E$ is given by $m_{k_1}$ equations
$F_{1,1},\dots,F_{1,m_1}$ of order $k_1$, \dots, $m_{k_s}$ equations
$F_{s,1},\dots,F_{s,m_s}$ of order $k_s=k_\text{max}$ ($s=r-m_{k_r}+1$).

We write $\E$ symbolically as $\sum\limits_{i=1}^rE_{k_i}=\sum m_iE_i$,
and call the latter the type of $\E$.
See \cite{K} for the table of class $\oo=1$ systems of order $k_\text{\rm max}\le 5$
(this table works equally well for general non-linear systems).

Let $g_i$ denote the symbols of $\E$. Starting from some jet-level $t$
the dimensions of these subspaces stabilize: $\dim g_i=1$ for $i\ge t$.
This is equivalent to involutivity of the prolongation $\E^{(t-k_\text{max})}$.

 \begin{dfn}
Complexity of $\E$ is the number $\varkappa=\sum_{i=0}^\infty(\dim g_i-1)$.
 \end{dfn}

This number measures the amount of Cauchy data needed to specify a solution.
It gives a partial order on the totality of class $\oo=1$ systems. All our reductions will
decrease the order.

By definition all systems of class $\oo=0$ (for ODEs the relevant complexity is
the dimension of the solutions space) are taken to be of lower complexity than the
systems of class $\oo=1$.

 \begin{lem}\label{Lem2}
Denote by $\hat\E$ the equation prolonged to the jet-level $t=\min\{i:\dim g_i=1\}$, where it is involutive.
Then $\dim\hat\E=\varkappa+t+3$.
 \end{lem}

 \begin{proof}
Since the base is 2-dimensional, we get $\dim\hat\E=2+\sum_{i=0}^t\dim g_i$, whence the claim.
 \end{proof}

Denote the Cartan distribution of $\E$ by $\CC_\E$.
We will assume that the system is prolonged to the level it is involutive, then $\dim\CC_\E=3$.

According to Lemma \ref{LemChar} the distribution $\CC_\E$ vcontains a
Cauchy characteristic field $\xi$ (the 1-dimensional distribution generated by it is unique).
Denote by $(M,\Delta)$ be the (local) quotient by $\xi$.
This $\Delta$ is a rank 2 distribution describing the internal geometry of $\E$.
By Lemma \ref{Lem2} the manifold $M$ has dimension $\mu=\varkappa+t+2$.

%%%%%%%%%%%%
\subsection{Derived flags of a rank 2 distribution}\label{S32}

Consider the strong derived flag of $\Delta$ defined by $\nabla_1=\Delta$,
$[\Gamma(\nabla_{i+1})=[\Gamma(\nabla_i),\Gamma(\nabla_i)]$ (where $\Gamma(\Delta)$
denotes the module of sections of the distribution $\Delta$).
The strong growth vector is the finite sequence of dimensions $(\dim\nabla_i)_{i=1}^\t$,
where $\t$ is the stabilization level (in the regularity assumptions, we adopt, all the ranks are constant).

The weak derived flag is given by $\Delta_1=\Delta$, $\Gamma(\Delta_{i+1})=[\Gamma(\Delta_i),\Gamma(\Delta)]$.
Notations $\p^{i-1}\!\Delta=\Delta_i$ are also used. The following cases are possible.

\medskip

{\bf I.} The growth vector is $(2,3,4,\dots)$. In this case by Cartan theorem \cite{C$_2$,AK}
the system can be de-prolonged\footnote{In \cite{AK} the growth vector of the weak derived flag
was considered. However this makes no difference at the first 3 elements of the
sequence $(2,3,x,\dots)$, where $x=\dim\Delta_3$ can be $3,4,5$.},
i.e. there exists another manifold $\bar M$
of dimension $\bar\mu=\mu-1$ equipped with rank 2 distribution
$\bar\Delta$ such that $\Delta=\mathbb{P}(\bar\Delta)$ is the prolongation.

The symmetries of $\Delta$ are preserved under passage to $\bar\Delta$, and the solutions
are mapped forward in such a way that to any solution of $\bar\Delta$ there corresponds
a 1-dimensional family of integral curves of $\Delta$.

Thus passage to de-prolongation is a nice reduction of the system, for which the complexity
$\varkappa$ (it exists on both ODE and PDE levels) decreases. For linear class $\oo=1$ systems this
corresponds to the (generalized) Laplace transformation, see \cite{K} and the next section.

\medskip

{\bf II.} The distribution $\Delta$ is not completely non-holonomic, i.e. $\nabla_\t\ne TM$.
In this case\footnote{Again here it makes no difference if we consider weak or strong
derived flag, only the length $\t$ can change.} there are $p=\mu-\op{rank}\nabla_\t$
first integrals $I_1,\dots,I_p$ that pull-back to first integrals of
the system $\E$. We can fix the values of $I_j$ and reduce the complexity of the system.

For linear systems existence of such integrals means that the sequence of Laplace transformations
does not reduce $\E$ to $E_1$ but stops on a finite type (class $\oo=0$) system \cite{K}.
For non-linear systems the relative invariants that control existence of intermediate integrals
can be calculated as generalized Laplace invariants of the linearization.

\medskip

{\bf III.} The general case: the distribution is totally non-holonomic and not de-prolongable.
Thus the growth vector is $(2,3,5,\dots,\mu)$.

To find integral curves of $\Delta$ one can evolve Theorem \ref{Thm1}, or use integrable extension
idea of Section \ref{S2} to decrease the complexity. Of course, due to calculations in
Section \ref{S22} a generic rank 2 distribution on high-dimensional manifold has no integrable extensions,
but distributions with symmetries do have such extensions. Indeed the symmetry reduction gives the projection
(with fibers almost everywhere transversal to $\Delta$), that's why we can treat integrable extensions
as generalized symmetries. Thus search of integrable extension is an integration method.

Notice also that the quotient by Cauchy characteristic is not interchangable with
de-prolongations and restrictions to the level of first integrals, and it must be performed first.
But sometimes the system needs to be prolonged for this.

\medskip

{\bf Example:} Consider a compatible system $\E\subset J^3(M)$ of type $2E_3$:
$u_{xxx}=F,u_{xyy}=G$, with $F,G\in C^\infty(J^2M)$).
The system is not involutive on the level of 3rd jets\footnote{It has non-zero Spencer cohomology
$H^{3,2}(\E)=\R$, and the symbol is not stable: $\dim g_3=2$, $\dim g_{3+i}=1$.}, so if we do not prolong the system, then
the weak derived flag of the Cartan distribution $\mathcal{C}_\E$ is
$(4,7,9,10)$. In addition $\mathcal{C}_\E$ has no Cauchy characteristics, while
its derived $\partial\,\mathcal{C}_\E$ has 3 Cauchy characteristics,
so that the pattern is wrong.

The prolonged system $\E^{(1)}\subset J^4(M)$ is involutive and the reduction of Theorem \ref{thA}
works -- there is 1 Cauchy characteristic for the original distribution and one more for the
derived.

Consider for instance the system with $F=\frac14u_{yyy}^4$, $G=\frac12u_{yyy}^2$.
The weak derived flag of the reduced (by Cauchy characteristic) system has growth $(2,3,4,5,6,7,8,9)$,
while the strong growth vector is $(2,3,4,6,9)$. Thus there is one intermediate integral
$u_{xxy}-\frac13u_{yyy}^3=c$, and after de-prolongation both growth vectors are
$(2,3,5,8)$ -- the corresponding graded nilpotent Lie algebra \cite{T,AK} is free truncated.

\smallskip

This shows importance of prolongation of the system $\E$ to the jet-level with $\dim g_k=1$.

%%%%%%%%%%%%
\subsection{Internal geometry of linear systems}\label{S33}

Linear compatible PDE systems of class $\oo=1$ were studied in \cite{K},
which we briefly summarize.

It was shown in that paper that such systems $\E$ (with dependent variable $u$)
can be integrated via generalized Laplace transformation, which is a first order differential
operator $L:u\mapsto v=Xu$, with $X$ having the same symbol as the characteristic
vector field.

Denote the system we obtain on the variable $v$ by $\ti\E$. It is also linear and compatible.
Denote the inverse operator by $L^{-1}:v\mapsto u$.
As proven in \cite{K} only three different situations are possible:

 \begin{enumerate}
 \item $\ti\E$ has class $\oo=1$ and $L^{-1}$ is a differential operator.
 \item $\ti\E$ has class $\oo=1$ but $L^{-1}$ is given by a finite type system.
 \item $\ti\E$ has class $\oo=0$ and $L^{-1}$ is an integral operator.
 \end{enumerate}

Case (1) is generic. If the itinerary of the transformations for $\E$ meets only
such equations, then Laplace transformation provides complete integration of the
PDE system $\E$.

Moreover, under generalized Laplace transformation the complexity strictly decreases.
Generically it decreases only by 1: $\varkappa\mapsto\varkappa-1$.

These results were obtained using the external geometry of $\E$.
Let us reformulate them in the internal language.

 \begin{prop}
A generalized Laplace transformation for $\oo=1$ li\-near systems is composed from the following
maps in subsequent stages: some number of prolongations, a diffeomorphism, some number of
de-prolon\-gations. For an involutive system only two last steps are required.
 \end{prop}

 \begin{proof}
Indeed, from internal viewpoint the rank 2 distribution $\Delta$, obtained from $\E$ via
reduction by Cauchy characteristic, is a Goursat distribution (or Goursat in the leaves
of the 1st integrals if the distribution is not totally non-holonomic).
Since the Goursat distribution has the canonical normal form
(see \cite{Ku,Mo}, we neglect singularities) the claim follows.
 \end{proof}

Let us show how this works. We start with generic linear $3E_3$ of class $\oo=1$.
Then in three Laplace transformations it becomes equation of type $E_1$
(we refer to \cite{K} for particular examples). We indicate the growth vector consisting of 
ascending by 1 integers, and indicate the internal
coordinates on the equation: $p,q,r,s,t$ are the classical notations for the 1st and 2nd derivatives,
and $\varrho$ is one of the 3rd derivatives).

%\fbox{\parbox{20pt}{AuA\\ A}}
%\vbox{\hrule\hbox{\vrule\,\strut A\,\vrule} \hbox{A}\hrule}

 $$
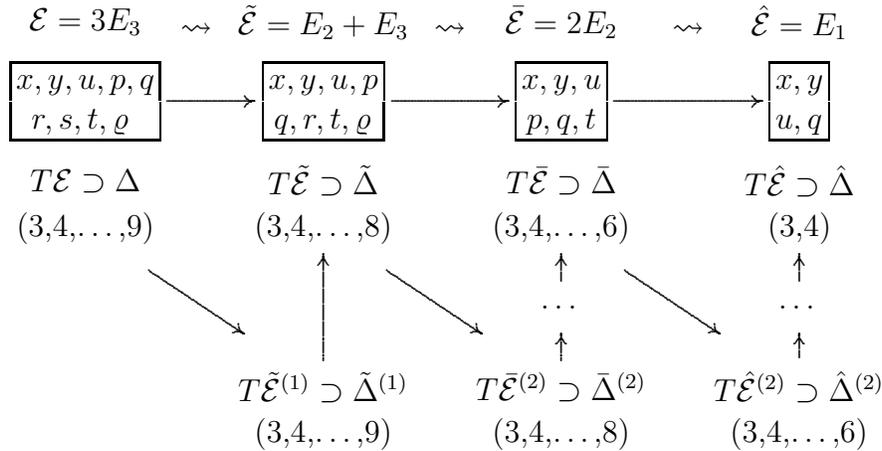
\begindc{\commdiag}[3]
 \obj(5,70)[17]{$\E=3E_3$}
 \obj(35,70)[27]{$\ti\E=E_2+E_3$}
 \obj(65,70)[37]{$\bar\E=2E_2$}
 \obj(95,70)[47]{$\hat\E=E_1$}
 \obj(19,69){$\rightsquigarrow$}
 \obj(51,69){$\rightsquigarrow$}
 \obj(81,69){$\rightsquigarrow$}
 \obj(5,60)[16]{\vbox{\hrule\hbox{\vrule\,\strut $x,y,u,p,q$\,\vrule}
      \hbox{\vrule\ \strut\ $r,s,t,\varrho$\hskip11.9pt\vrule}\hrule}}
 \obj(35,60)[26]{\vbox{\hrule\hbox{\vrule\,\strut $x,y,u,p$\,\vrule}
      \hbox{\vrule\ \strut $q,r,t,\varrho$\hskip4.6pt\vrule}\hrule}}
 \obj(65,60)[36]{\vbox{\hrule\hbox{\vrule\,\strut $x,y,u$\,\vrule}
      \hbox{\vrule\ \strut $p,q,t$\hskip3.9pt\vrule}\hrule}}
 \obj(95,60)[46]{\vbox{\hrule\hbox{\vrule\,\strut $x,y$\,\vrule}
      \hbox{\vrule\,\strut $u,q$\hskip2.6pt\vrule}\hrule}}
% \obj(40,50){\fbox{\parbox{10pt}{A}}}
% \obj(0,50){\fbox{\hbox{$x,y,u,u_x,u_y,$}\hbox{$u_{xy},u_{yy},u_{yyy}$}}}
 \mor{16}{26}{} \mor{26}{36}{} \mor{36}{46}{}
 \obj(5,50)[15]{$T\E\supset\Delta$}
 \obj(35,50)[25]{$T\ti\E\supset\ti\Delta$}
 \obj(65,50)[35]{$T\bar\E\supset\bar\Delta$}
 \obj(95,50)[45]{$T\hat\E\supset\hat\Delta$}
 \obj(5,44)[14]{(3,4,\dots,9)}
 \obj(35,44)[24]{(3,4,\dots,8)}
 \obj(65,44)[34]{(3,4,\dots,6)}
 \obj(95,44)[44]{(3,4)}
 \obj(35,24)[23]{$T\ti\E^{(1)}\supset\ti\Delta^{(1)}$}
 \obj(35,18)[22]{(3,4,\dots,9)}
 \mor{14}{23}{} \mor{23}{24}{}
 \obj(65,24)[33]{$T\bar\E^{(2)}\supset\bar\Delta^{(2)}$}
 \obj(65,34)[3a]{\dots}
 \obj(65,18)[32]{(3,4,\dots,8)}
 \mor{24}{33}{} \mor{33}{3a}{} \mor{3a}{34}{}
 \obj(95,24)[43]{$T\hat\E^{(2)}\supset\hat\Delta^{(2)}$}
 \obj(95,34)[4a]{\dots}
 \obj(95,18)[42]{(3,4,\dots,6)}
 \mor{34}{43}{} \mor{43}{4a}{} \mor{4a}{44}{}
\enddc
 $$

The first transformation is a diffeomorphism followed by a de-pro\-lon\-gation,
the next one is a diffeomorphism followed by two de-prolonga\-tions
and the third one is of the same kind.

If we choose non-generic $3E_3$, then the route could be
$3E_3\rightsquigarrow 2E_2\rightsquigarrow E_1$, so that
the first stage contains more de-prolongations.

Starting with $2E_3$ one needs to prolong once to follow the scheme.

 \begin{rk}
Now we can explain decrease of complexity $\varkappa$ via internal geometry.
Since $\varkappa=\dim\E-t-3$, where $t$ is the order of involutivity, we see
that $\varkappa$ is defined correctly even if we prolong above the involutivity level
(increase $\dim\E$ and $t$ equally). The diffeomorphism in the above proposition
does not change the dimension, but it increases the order. Whence the claim.
 \end{rk}

%%%%%%%%%%%%
\subsection{Closed form of the general solution}\label{S34}

The shift along Cauchy characteristic is a characteristic symmetry for $\E$,
so its flow is tangent to solutions of this system (induces the trivial vector
field on $\op{Sol}(\E)$). Thus $\E$ has closed form of the general solution iff
the same is true for the reduced underdetermined ODE (encoded by $\Delta$).

Recall that closed form means possibility to represent general solution parametrically
through arbitrary functions and free parameters, which for PDE systems of class $\oo=1$ writes
 \begin{equation}
(x,y,u)=\Psi(\varsigma,f(\tau),f'(\tau),\dots,f^{(q)}(\tau)).
 \end{equation}
Here $\Psi$ is a $\R^3$-valued function with Jacobian of rank 3.
For underdetermined ODE we shall remove one independent variable to the left,
and the free variable $\varsigma$ to the right.

For rank 2 distributions the criterion for closed form description of the
integral curves is known since E.Cartan. Namely in \cite{C$_2$,Ku} this
is shown to be equivalent to $\Delta$ being Goursat, i.e. the canonical
distribution $\CC$ on the jet-space $J^d(\R,\R)$, $d=\mu-2=\varkappa+t$.

On the other hand we have demonstrated in \cite{K} that systems that are
internally equivalent to $(J^d(\R,\R)\times\R,\CC\times\R)$
are internally linearizable and have no intermediate integrals.

Intermediate integrals correspond to constants in the form of the general solution
 \begin{equation}
(x,y,u)=\Psi(\varsigma,f(\tau),f'(\tau),\dots,f^{(q)}(\tau),c_1,\dots,c_m).
 \end{equation}
The distribution can be transformed to Goursat-Frobenius normal form, namely
it is internally equivalent to $(J^{d-m}(\R,\R)\times\R^m\times\R,\CC\times0\times\R)$, where
$m=\op{codim}(\Delta_\infty)$ is the codimension of the bracket-closure of $\Delta$.

Linearizability is not hampered by the additional constants.
Thus we get the following statement.

 \begin{theorem}
General solution of a generic compatible system $\E$ of class $\oo=1$
and complexity $\varkappa$ can be expressed in a closed form via a function
$f$, its $q\le\varkappa$ derivatives and $\varkappa-q$ constants iff $\E$
is linearizable by an internal transformation.
 \end{theorem}

 \begin{proof}
Possibility to express solutions of a linear compatible systems of class $\oo=1$
in closed form is proved in \cite{K}, so we need only to demonstrate that $q+m=\varkappa$,
 % (in analytic situation this can be seen by comparison of the initial data for the Cauchy problem)
where $m$ is the amount of the first integrals (constants).

The amount of derivatives to express all internal coordinates
$u_\z$, $|\z|\le t$, on $\E$ is $d-m$. However the derivatives $u_\z$ are obtained from
$(x,y,u)$ via $|\z|$ differentiations, so $u$ shall be expressed in $d-m-t$ derivatives of $f$
only, and this number equals to $\varkappa-m$.
 \end{proof}

%%%%%%%%%%%%
\subsection{Transformations of non-linear systems}\label{S35}

Let us discuss some features of the transformations theory in the non-linear case.

\smallskip

{\bf I.} Quasi-linear systems allow some de-prolongations, but generic
pure order $k$ systems ($kE_k$) have none --- after quotient by Cauchy characteristics the
growth vector is $(2,3,5,\dots)$. However if $\E$ is involutive ($t=k_\text{max}$)
with different orders ($k_\text{min}<k_\text{max}$), then we claim: {\em The top equations
are quasi-linear, and this implies existence of at least one de-prolongation}.

Indeed, provided that the characteristic is $\p_x-\l\,\p_y$, where $\l$ is a function on
the jets of order $k_\text{min}$, the top derivatives on the level $k=k_\text{max}$ must satisfy
$u_{i,k-i}=\l^iu_{0,k}$ (this is due to the fact that $\xi$ is the characteristic for all PDEs
of $\E$). Thus the PDE of order $k_\text{max}$ in $\E$ can be chosen linear in top-derivatives.

Consequently $\D_x-\l\D_y+\rho\p_{u_{0,k}}$ is the Cauchy characteristic of $\CC_\E$ for some
function $\rho$ on $k$-jets, and the two other generators of $\CC_\E$ are $\D_y$ and $\p_{u_{0,k}}$.
A straightforward calculation yields that the latter field is a Cauchy characteristic for
the derived distribution $\p\,\CC_\E$, so the system can be de-prolonged.

\medskip

{\bf II.} Re-covering the first integrals is the same as for the linear systems.
These restrictions introduce constants to the form of the general solution of $\E$,
similar as de-prolongations add derivatives to the form of the general solution.
The two latter operations commute with each other and also with the projection of
an integrable extension.

\medskip

{\bf III.} If $\E$ allows the structure of $\tilde d$-dimensional integrable extension
$\E\to\tilde\E$, then its solutions can be expressed via those of $\tilde\E$ as
$u=L_{\tilde d}(\tilde u)$, where $L_{\tilde d}$ is the resolution
operator of a scalar ODE of order $\tilde d$.

Provided a solvable Lie group of dimension $\tilde d$ acts by transversal symmetries,
the operator $L_{\tilde d}$ can be expressed via the $\tilde d$-multiple
quadrature $\D_\tau^{-\tilde d}$ ($\D_\tau^{-1}=\int\square\,d\t$ being
quadrature by the parameter $\t$).

For instance, if $\Delta$ is maximally symmetric non-Goursat distribution, then its de-prolongation
is flat in the sense of Tanaka \cite{T}, and so has the structure of
successive integral extensions over the rank 2 distribution in $\R^5$ with $G_2$ symmetry
\cite{AK}. So we get

 \begin{theorem}
If $\E$ has reduction $\Delta$, which de-prolongs to a Tanaka-flat rank 2 distribution,
then $\E$ can be solved in closed form and quadratures.
 \end{theorem}

Indeed, de-prolongations can be interpreted as non-linear Laplace transformations with
differential inverses (this yields a closed form over the solutions of the reduced system $\bar\Delta$),
while Tanaka flat rank 2 distributions $\bar\Delta$ project via integrable extensions
to the Hilbert-Cartan equation \cite{AK} (so its integral curves are given in quadratures).

Thus we get the next easy case (after linearizable systems) of explicitly integrable class $\oo=1$
systems, which are reduced to symmetric Monge systems (these latter were classified in \cite{AK}).

 \begin{rk}
If $\E$ is an integrable extension, then its solution form can be specified via particular
quadratures, like it is done in \cite{C$_1$,G$_3$} in the case of rank 2 distribution in dimension 5.
 \end{rk}

Finally let us mention that the solutions obtained via the proposed method are usually
not of Moutard type, where the latter means that $u$ is expressed directly as a function
of $x,y$ (in the above approach all $x,y,u$ are expressed as functions of the additional
parameters $\varsigma,\t$).

Indeed in order for the discussed transformations to be Moutard, they
shall preserve the vertical fibrations, i.e. the Cauchy characteristics for the derived
systems and the generalized symmetries must be tangent to the vertical fibers
(the 1st integrals are always independent of the base coordinates $x,y$).

%%%%%%%%%%%%
\section{Examples of symmetric PDEs}\label{S4}

%%%%%%%%%%%%
\subsection{Model reductions to ODEs}\label{S41}

Consider the following compatible class $\oo=1$ systems of the type $kE_k$:
 \begin{gather*}
2E_2:\ u_{xx}=\l,\ u_{xy}=\frac{\l^2}2,\ u_{yy}=\frac{\l^3}3;\\
3E_3:\ u_{xxx}=\l,\ u_{xxy}=\frac{\l^2}2,\ u_{xyy}=\frac{\l^3}3,\ u_{yyy}=\frac{\l^4}4;\\
4E_4:\ u_{xxxx}=\l,\ u_{xxxy}=\frac{\l^2}2,\ u_{xxyy}=\frac{\l^3}3,\ u_{xyyy}=\frac{\l^4}4,\ u_{yyyy}=\frac{\l^5}5\ \
\text{etc.}
 \end{gather*}
The reduced growth vectors\footnote{These are the ones we have used in \cite{AK}:
we pass from the usual growth vector $(n_1,n_2,n_3,\dots)$ to $(n_1,n_2-n_1,n_3-n_2\dots)$.}
and the generators of the weak derived flags are the following:
 \begin{gather*}
(2,1,2):\ (e_1,e_1',e_2,e_3,e_3'),\\
(2,1,2,3):\ (e_1,e_1',e_2,e_3,e_3',e_4,e_4',e_4''),\\
(2,1,2,3,4):\ (e_1,e_1',e_2,e_3,e_3',e_4,e_4',e_4'',e_5,e_5',e_5'',e_5''')\ \
\text{etc.}
 \end{gather*}
Commutators are given by $[e_1,e_1']=e_2$, $[e_1,e_2]=e_3$, $[e_1',e_2]=e_3'$,
$[e_1,e_3]=e_4$, $[e_1,e_3']=[e_1',e_3]=e_4'$, $[e_1',e_3']=e_4''$ etc
(the commutators of $e_2$ and $e_3$ and others are zero) ---
these yields the structure of graded nilpotent Carnot algebra associated to the weak derived flag \cite{T,AK}.

The corresponding Monge underdetermined systems of ODEs are:
 \begin{gather*}
(2,1,2):\ \ y'=\tfrac12(z'')^2,\\
(2,1,2,3):\ \ y''=\tfrac12(z''')^2,\ u'=\tfrac13(z''')^3,\\
(2,1,2,3,4):\ \ y'''=\tfrac12(z^{iv})^2,\ u''=\tfrac13(z^{iv})^3,\ v'=\tfrac14(z^{iv})^4\ \
\text{etc.}
 \end{gather*}

This follows from the explicit form of the generators.
Indeed, let us demonstrate this, for simplicity, in the case $3E_3$.

 \begin{gather*}
e_1=-\D_x=-(\p_x+u_x\p_u+u_{xx}\p_{u_x}+u_{xy}\p_{u_y}+
 \l\p_{u_{xx}}+\tfrac{\l^2}{2}\p_{u_{xy}}+\tfrac{\l^3}{3}\p_{u_{yy}}),\\
e_1'=\p_t,\ \ e_2=\p_{u_{xx}}+\l\p_{u_{xy}}+\l^2\p_{u_{yy}},\\
e_3=\p_{u_x}+\l\p_{u_y},\ e_3'=\p_{u_{xy}}+2\l\p_{u_{yy}},\\
e_4=\p_u,\ e_4'=\p_{u_y},\ e_4''=2\p_{u_{yy}}.
 \end{gather*}
In this list we have omitted the generator of (rank 3) Cartan distribution $\D_y$ because
the Cauchy characteristic equals $\D_y-\l\D_x$ and we need to quotient by it.

Now to perform the quotient one has either to pass to invariants, or to restrict
to the transversal of the action. We choose the second approach: Change the notations
$z:=u$, $z':=u_x$, $z'':=u_{xx}$, $z''':=\l=u_{xxx}$, $y:=u_y$, $y':=u_{xy}$, $u:=u_{yy}$ and we are done.

 \begin{rk}
To recover the solutions of $\E$ from the solutions of the ODE system,
the first method must be used. For example, the first of the equations --- system
$2E_2$ from \cite{C$_1$} --- has the following invariants of shifts along Cauchy characteristic ($u_{xx}=\l$)
 \begin{multline*}
t=-\l,\
w=-2u+2y u_y+2x u_x-x^2\l-xy\l^2-\tfrac13y^2\l^3,\\
z=u_y-u_x\l+\tfrac12 x\l^2+\tfrac16y\l^3,\
z_1=u_x-x\l-\tfrac12y\l^2,\ z_2=x+y\l.
 \end{multline*}
Thus we can express the general solution parametrically as
 $$
x=z''(t)+st,\ y=s,\
u=sz+z'z''-\tfrac12w-\tfrac12t\,z''^2-\tfrac12t^2s\,z''-\tfrac16t^3s^2,
 $$
where the two functions $z=z(t),w=w(t)$
are related by the Hilbert-Cartan equation $w'=(z'')^2$.
 \end{rk}

Another interesting sequence of equations, considered in \cite{G$_1$}, is provided by
 \begin{equation}\label{Goursat}
u_{xy}=\frac{2n}{x+y}\sqrt{u_xu_y}.
 \end{equation}
These PDEs are Darboux integrable with intermediate integrals of order $(n+1)$, see \cite{AF}.

In order to integrate (\ref{Goursat}) let us linearize it via Goursat substitution \cite{G$_2$}
$p=\sqrt{u_x}$, $q=\sqrt{u_y}$, which leads to the system
 \begin{equation}\label{Goursa1}
p_y=\frac{n}{x+y}q,\quad q_x=\frac{n}{x+y}p.
 \end{equation}
Then $\D_y$-intermediate integral can be found as an ODE on $q$ of the form
($\D_x$-intermediate integral is obtained similarly via an ODE on $p$)
 \begin{equation}\label{Goursa2}
L=\sum_{i=0}^n\frac{\alpha_i}{(x+y)^{n-i}}q_i=0,
 \end{equation}
where $q_i=\D_y^iq$. The condition of intermediate integral --- $L_x=0$ on (\ref{Goursa1}) ---
is an overdetermined linear system on constants $\a_i$. With normalization
$\a_n=1$ its unique solution is given by
 $$
\a_{n-i}=\frac{n^2(n-1)^2\cdots(n-i+1)^2}{i!}.
 $$

Consider now the overdetermined compatible system (\ref{Goursat})+(\ref{Goursa2}).
It has class $\oo=1$ and type $E_2+E_m$, $m=n+1$.
The Cauchy characteristic is $\D_y+\vp_m\p_{q_m}$
for a properly chosen function $\vp_m$.

Reduction along Cauchy characteristic (which again can be interpreted as intersection
with the level of the $\D_x$-intermediate integral) yields a rank 2 distribution on a manifold
$\bar\E$ of dimension $(2n+3)$. This system has $n$ de-prolongations, and so it reduces to a
rank 2 distribution on $(n+3)$-dimensional manifold.

The symmetry analysis (done by I.Anderson) coupled together with unique symmetry model
for rank 2 distributions \cite{DZ,AK} implies that this rank 2 distribution corresponds
to the Cartan distribution of the Monge equation
 $$
y'=(z^{(n)})^2.
 $$
In particular, we recover the result (known to Goursat) that the solutions of
equation (\ref{Goursat})$_{n=2}$ are expressed via the solutions of Hilbert-Cartan equation.

 \begin{rk}
Thus we realize two boundary lines from the Zoo of types in \cite{K} --- the bottom and the diagonal ---
as the most symmetric PDEs in its class (both types of PDE and the reduction are fixed).

Here we refer to the contact symmetries (internal symmetry group of class $\oo=1$ systems
is infinite-dimensional), which turns out to be isomorphic to the internal symmetry group
of the reduction by Cauchy characteristic in most cases. This Lie-B\"acklund type theorem
will be discussed in details elsewhere.
 \end{rk}

%%%%%%%%%%%%
\subsection{Representation of solutions}\label{S42}

Linear class $\oo=1$ compatible systems $\E$ have their solutions expressible in Moutard
form, i.e. $u$ is a parametrized function of $(x,y)$. In general, the closed form solution
force all of the variables $(x,y,u)$ to be expressed via parameters $(\zeta,\t)$ ---
this is equivalent to linearizability, see \S\ref{S34}.

Linearization map does not commute, in general, with the projection to the base $(x,y)$,
that's why many exactly solvable system do not possess Moutard form. We illustrate this with
two 2nd order examples.

\medskip

{\bf Example 1.} Consider the system (\ref{Goursat})+(\ref{Goursa2})$_{n=1}$:
 \begin{equation}\label{sqrt2}
u_{xx}=-2\frac{u_x}{x+y},\ u_{xy}=2\frac{\sqrt{u_xu_y}}{x+y}.
 \end{equation}
The reduced growth vector is $(3,1,1,1)$, so $\E$ is internally linearizable and is solvable via
generalized Laplace transformations \cite{K}. It is however non-Moutard.

To see this let us describe the general solution. Goursat substitution
 \begin{equation}\label{G-subst}
u_x=p^2,\ u_y=q^2,
 \end{equation}
linearizes the equation
 \begin{equation}\label{R45}
p_x=-\frac{p}{x+y},\ q_x=\frac{p}{x+y},\ p_y=\frac{q}{x+y}.
 \end{equation}
Notice that this has vector growth $(3,1,1)$, and so one could suggest it serves as de-prolongation of (\ref{sqrt2}),
but it does not. The reason (as shown below) is that it is impossible to de-prolong preserving the base coordinates
$x,y$. In fact, (\ref{sqrt2}) is an integrable extension of (\ref{R45}) via (\ref{G-subst}).

The next step is to observe that the last equation of (\ref{R45}) can be used as a definition of $q$ and the
second equation is then the differential corollary of the first. Thus we can restrict to the equation $E_1$:
 \begin{equation}\label{R46}
p_x=-\frac{p}{x+y}
 \end{equation}
with the reduced growth vector $(3,1)$. Inverse transformations are: integrable extension
 $$
u_x=p^2,\ u_y=(x+y)^2p_y^2
 $$
to the space $\R^5(x,y,u,u_x,u_y)$ and then prolongation to the original equation $\E=\R^6(x,y,u,u_x,u_y,u_{yy})$.

Let us return to the closed form of the solution. Solving (\ref{R46}) we get $p=\psi(y)/(x+y)$.
Integrating the Frobenius system
 $$
u_x=\frac{\psi(y)^2}{(x+y)^2},\ u_y=\Bigl(\psi'(y)-\frac{\psi(y)}{x+y}\Bigr)^2.
 $$
yields the solution
 \begin{equation}\label{uphipsi}
u=\phi(y)-\frac{\psi(y)^2}{x+y}\ \text{ with }\ \phi'(y)=\psi'(y)^2.
 \end{equation}
This latter constraint is internally equivalent to the Engel distribution (or to $J^2(\R,\R)$)
and the equivalence is given explicitly by
 $$
y=\z''(\t),\ \psi=\t\z''(\t)-\z'(\t),\ \psi'=\t,\ \phi=\t^2\z''(\t)-2\t\z'(\t)+2\z(\t).
 $$
Thus alignment of the equation $\E$ (\ref{sqrt2}) to the jet-space $J^{0,3}(\R,\R^2)=\R^6(x,\t,\z,\z_1,\z_2,\z_3)$
is the following:
 \begin{multline*}
x=x,\ y=\z_2,\ u=\t^2\z_2-2\t\z_1+2\z-\frac{(\t\z_2-\z_1)^2}{x+\z_2},\\
u_x=\frac{(\t\z_2-\z_1)^2}{(x+\z_2)^2},\ u_y=\frac{(\t x+\z_1)^2}{(x+\z_2)^2},\\
u_{yy}=2(\t x+\z_1)\frac{((x+\z_2)^2-(\t x+\z_1)\z_3)}{(x+\z_2)^3\z_3}.
 \end{multline*}
This determines a diffeomorphism $J^{0,3}(\R,\R^2)\to\E$. The inverse map $\E\to J^{0,3}(\R,\R^2)$
is also given in differential-algebraic form
 \begin{multline*}
x=x,\ \t=\sqrt{u_x}+\sqrt{u_y},\ \z=\frac12(u-x u_x+y u_y-2x\sqrt{u_xu_y}),\\
\z_1=y\sqrt{u_y}-x\sqrt{u_x},\ \z_2=y,\ \z_3=\frac{2(x+y)\sqrt{u_y}}{2u_y+(x+y)u_{yy}}.
 \end{multline*}
Thus Laplace transformation $\E=2E_2\mapsto E_1=\{\frac{\p\z}{\p x}=0\}$ has the form
 $$
x=x,\ \t=\sqrt{u_x}+\sqrt{u_y},\ \z=\frac12(u-x u_x+y u_y-2x\sqrt{u_xu_y})
 $$
and it decomposes internally into the composition of the diffeomorphism $\E\to J^{0,3}(\R,\R^2)$
followed by the double de-prolongation (we do not change the original Cauchy characteristic, which
is the direction of the first factor in $J^{0,3}(\R,\R^2)=\R\times J^3(\R,\R)$):
 $$
J^{0,3}(\R,\R^2)\to J^{0,1}(\R,\R^2)=\R^4(x,\t,\z,\z_1).
 $$

Let us explain why $2E_2$ (\ref{sqrt2}) cannot be transformed to $E_1$ by a Laplace transformation
preserving the $(x,y)$-base (Moutard type).

We wish to find a relation on $x,y,u,u_x,u_y$ excluding the above functions
$\phi(y),\psi(y)$. But this is impossible since
 $$
\sqrt{u_x}=\frac{\psi(y)}{x+y},\ \sqrt{u_y}=\psi'(y)-\frac{\psi(y)}{x+y}
 $$
and so $u,u_x,u_y$ are algebraically independent.

Another approach is to show that the constraint $\phi'(y)=\psi'(y)^2$ in the form
(\ref{uphipsi}) is equivalent to the standard Engel distribution on $J^2(\R,\R)$
internally, but not externally. Indeed, no point transformation can map the above
constraint to the equation $\phi'(y)=0$, since their point symmetry groups have
dimensions $10$ and $\infty$ respectively.

\medskip

{\bf Example 2.} Another interesting system, discussed in \S\ref{S41}, is
the Cartan involutive $2E_2$ model
 \begin{equation}\label{ECeq}
u_{xx}=\l,\ u_{xy}=\tfrac12\l^2,\ u_{yy}=\tfrac13\l^3.
 \end{equation}
The reduced growth vector of its (5-dimensional) reduction $\bar\E$ is $(2,1,2)$,
so Laplace transformation in the sense of linear theory does not exist.
But $\E$ has the structure of integrable extension over $E_1$, namely over the gas dynamics equation
 \begin{equation}\label{R33}
v_y=v\,v_x.
 \end{equation}
Indeed, this latter is just the compatibility condition on the parameter $\l=v$ along the Cauchy characteristic.
The transformation from (\ref{R33}) to (\ref{ECeq}) is a composition of a 3-dimensional integrable extension
and the de-prolongation:
 $$
\R^6(x,y,u,u_x,u_y,u_{yy})\stackrel{\text{prol}}\dashrightarrow
\R^7(x,y,u,u_x,u_y,u_{yy},u_{yyy})\stackrel{\int\text{ext}}\dashleftarrow\R^4(x,y,v,v_y).
 $$
Since (\ref{R33}) is clearly not of Moutard type, this explains that (\ref{ECeq}) is not of Moutard type.
Its solutions are though expressible via the solutions of Hilbert-Cartan equation.

This makes (\ref{ECeq}) internally equivalent to de-prolongation of
the equation (\ref{Goursat})+(\ref{Goursa2})$_{n=2}$ of type $E_2+E_3$ considered in \S\ref{S41}.
Another equation equivalent to (\ref{ECeq}) is thus the compatible $2E_2$:
 $$
w_{xx}=0,\ w_{xy}^2+xw_{xy}-w_y=0.
 $$

%%%%%%%%%%%%
\subsection{Other symmetric models}\label{S43}

In \cite{C$_1$} Cartan considers also sub-maximal symmetric systems $\E$:
 \begin{equation}\label{Ctsm1}
2E_2:\quad u_{xx}=\l,\ u_{xy}=\l^m,\ u_{yy}=\frac{m^2}{2m-1}\l^{2m-1}.
 \end{equation}
Its contact symmetry algebra has dimension 7, and this is the next possible number after the maximal
finite value 14 for $\dim\op{Sym}(\E)$.

The Cauchy characteristic is $\xi=\D_y-m\l^{m-1}\D_x$, and the reduced system $\bar\E$
can be found by restricting to the transversal $y=\op{const}$ to $\xi$.
In other words, the rank 2 distribution of $\bar\E$ is given by its generators
 $$
\Delta=\langle\D_x=\p_x+p\p_u+r\p_p+r^m\p_q,\p_r\rangle,
 $$
which after a change of coordinates is identical with the Cartan distribution of the Monge equation
($m\ne 2,\tfrac13$ -- the exceptional cases corresponding to $\dim\op{Sym}(\Delta)=14$)
 \begin{equation}\label{9o0}
w'=(v'')^m.
 \end{equation}

The higher analogs of (\ref{Ctsm1}) are straightforward:
 $$
3E_3:\ u_{xxx}=\a,\ u_{xxy}=\a^m,\ u_{xyy}=\frac{m^2}{2m-1}\a^{2m-1}, u_{yyy}=\frac{m^3}{3m-2}\a^{3m-2}
 $$
Dimension of the contact symmetry algebra here is 10 (this is readily checked
with the help of \textit{Differential Geometry} package of \textsc{Maple}),
while dimension of the maximal symmetric non-linear model (system $3E_3$ from \S\ref{S41}) is 12.

The reduction by Cauchy characteristic is the following underdetermined ODE system:
 $$
z''=(v''')^m,\ w'=\frac{m^2}{2m-1}(v''')^{2m-1}.
 $$
It is a (3-dimensional) integrable extension of Monge equation (\ref{9o0}).

Similarly we construct equations $4E_4$ etc. The sub-maximal symmetric class $\oo=1$ system of
type $nE_n$ is
 $$
\Bigl\{u_{n-i,i}=\frac{m^i}{im-i+1}u_{n,0}^{im-i+1}\,:\,1\le i\le n\Bigr\}.
 $$
Its contact symmetry algebra has dimension $\frac12n(n+1)+4$ (vs. the maximal dimension
$\frac12n(n+1)+6$ --- the details on this calculation will be presented elsewhere).

It is an interesting open problem what are the sub-maximal symmetric PDE systems of
the type $E_2+E_n$ and what are (sub-)maximal models for the other types from the Zoo of \cite{K}.

%%%%%%%%%%%%%%%%%%%%%%%%%%%%%%%%%%%%%%%%%%%%%%%%%%%%%%%%%%%%%%%%%%%%%%%%%%%%

\end{document}